\documentclass[11pt]{amsart} 

\usepackage{amsmath,amssymb,amsthm,amsfonts,verbatim,comment}
\usepackage{enumerate}
\usepackage{color}
\usepackage{graphicx}

\usepackage{tikz-cd}

\theoremstyle{plain}
\newtheorem{theorem}{Theorem}[section]
\newtheorem{question}[theorem]{Question}

\newtheorem{lemma}[theorem]{Lemma}

\newtheorem{remark}[theorem]{Remark}

\newtheorem{corollary}[theorem]{Corollary}

\theoremstyle{definition}
\newtheorem{definition}[theorem]{Definition}

\newtheorem*{convention*}{Convention}

\newcommand{\dmo}{\DeclareMathOperator}

\newcommand{\R}{\mathbb{R}}

\newcommand{\Z}{\mathbb{Z}}

\dmo{\Mod}{Mod}
\dmo{\Ig}{\mathcal{I}_g}
\dmo{\Span}{span}
\dmo{\Diff}{Diff}
\dmo{\Homeo}{Homeo}
\dmo{\dist}{dist}
\dmo\BDiff{BDiff}
\dmo\SO{SO}
\dmo\slide{sl}
\dmo\im{im}
\dmo\id{id}
\dmo\Fix{Fix}
\dmo\scl{scl}
\dmo\Stab{Stab}
\dmo\Mcg{Mcg}
\dmo{\Hg}{\mathcal{H}_g}
\dmo{\Tg}{\mathcal{T}_g}

\renewcommand{\epsilon}{\varepsilon}
\newcommand{\coloneq}{\mathrel{\mathop:}\mkern-1.2mu=}

\newcommand{\nc}{\mathcal{NC}}

\newcommand{\cd}{\mathcal{C}^\dagger}

\title{Quasi-morphisms on surface diffeomorphism groups}
\dedicatory{Dedicated to Mladen Bestvina on the occasion of his 60\textsuperscript{th} birthday.}
\author{Jonathan Bowden}
\address{School of Mathematical Sciences, 
Monash University
Victoria, 3800, Australia}
\email{jonathan.bowden@monash.edu}
\author{Sebastian Hensel}
\address{Mathematisches Institut der Universit\"at M\"unchen,
Theresienstra\ss{}e 39, 80333 M\"unchen}
\email{hensel@math.lmu.de}
\author{Richard Webb}
\address{Department of Mathematics,
The University of Manchester,
Oxford Road,
Manchester,
M13 9PL}
\email{richard.webb@manchester.ac.uk}

\usepackage[
pdfborderstyle={},
pdfborder={0 0 0},
pagebackref,
pdftex]{hyperref}

\begin{document} 
\maketitle

\begin{abstract}
We show that the identity component of the group of diffeomorphisms of a closed oriented surface of positive genus admits many unbounded quasi-morphisms. As a corollary, we also deduce that this group is not uniformly perfect and its fragmentation norm is unbounded, answering a question of Burago--Ivanov--Polterovich. As a key tool we construct a hyperbolic graph on which these groups act, which is the analog of the curve graph for the mapping class group.

\end{abstract}

\section{Introduction} 

The general problem of constructing quasi-morphisms has played a
prominent role in geometric group theory, symplectic geometry and
dynamics since Gromov's introduction of bounded cohomology in the
early 80s \cite{Gro}. In the context of mapping class groups 
there are now many constructions of quasi-morphisms, starting with the 
work of Endo--Kotschick \cite{EnKo}, Bestvina--Fujiwara \cite{BF}, and
Hamenst\"adt \cite{Ham} to name only a few. Using the natural projection 
of the full diffeomorphism group to the
mapping class group this then yields many non-trivial quasi-morphisms on the
{\em full} diffeomorphism group $\Diff(S_g)$ of any surface
$S_g $ of genus $g \ge 2$.

In view of this it remains to study whether the identity component
$\Diff_0(S_g)$ admits any non-trivial quasi-morphisms.

If one restricts to the subgroup of area-preserving diffeomorphisms, or more precisely the
subgroup of Hamiltonian diffeomorphisms, there are constructions of quasi-morphisms
due to Ruelle \cite{Ru} and Gambaudo--Ghys \cite{GaGh} (see also \cite{Py, BM}). This motivates the following question going back to Burago--Ivanov--Polterovich \cite{BIP}:
\begin{question}
Does the group $\Diff_0(S_g)$ admit any quasi-morphisms that are unbounded?
\end{question}
To put this question in context, recall that for any compact manifold
the identity component of the diffeomorphism group
$\Diff_0(M)$ is perfect by classical results of Mather and
Thurston \cite{M1,M2,T}\ and therefore does not admit any non-trivial
homomorphisms to an abelian group (cf.\ also \cite{Mann}). In fact, 
combining results of Burago--Ivanov--Polterovich \cite{BIP} and Tsuboi \cite{Tsu, Tsu_even}, for any closed, oriented manifold $M$ which is a sphere, or has dimension $\dim M \neq 2$ or $4$, the group $\Diff_0(M)$ is uniformly perfect: any element can be written as a product of commutators of uniformly bounded length. It is easy to see that this is an obstruction to the existence of an unbounded quasi-morphism.

\smallskip The simplest manifolds that are not covered by these general results
are the closed surfaces of genus $g\geq 1$. Our main result shows that these groups indeed show drastically different behaviour: 

\begin{theorem}\label{intro:main-theorem}
  For $g\geq 1$ the space $\widetilde{\mathrm{QH}}(\Diff_0(S_g))$ of
  unbounded quasi-morphisms on $\Diff_0(S_g)$ is
  infinite dimensional.
\end{theorem}
\noindent Here $\widetilde{\mathrm{QH}}(G)$ denotes
the space of quasi-morphisms modulo the set of bounded functions
on $G$ (see Section~\ref{sec:quasi-morphisms} for further details).

The existence of an single unbounded quasi-morphism on
$\Diff_0(S_g)$ already has the following consequence, which
answers a question of Burago--Ivanov--Polterovich \cite{BIP}:

\begin{corollary}\label{cor:frag}
  For $g \geq 1$ the group $\Diff_0(S_g)$ is not uniformly perfect and has unbounded fragmentation norm.
\end{corollary} 
\noindent Note that by the main result of \cite{BIP} the fragmentation norm on $\Diff_0(S^2)$ is uniformly
bounded and so the assumptions in the corollary are optimal (compare Section~\ref{sec:spherefail} on why our method does not apply to $\Diff_0(S^2)$).  

%


\subsection*{Outline of Proof:} 

The quasi-morphisms in Theorem~\ref{intro:main-theorem} are built via
group actions on hyperbolic spaces using a general method 
due to Bestvina--Fujiwara \cite{BF} (see also \cite{EpsteinFujiwara, Fuji}). 

We will apply this to a variant of the \emph{curve graph}.
Curve graphs were originally defined by Harvey \cite{Harvey} and, due to the foundational work of
Masur and Minsky \cite{MM1, MM2},  have quickly become one of the central tools to
study mapping class groups. In particular, Masur and Minsky were the first to show that the curve graph is hyperbolic
\cite{MM1}.
Many variants of curve graphs have since then been used to study subgroups of mapping class groups
and their relatives. We just want to highlight that, following a strategy suggested by Calegari, 
curve graphs have been used to construct
quasi-morphisms on \emph{big mapping class groups}, e.g. the
mapping class group of the plane minus a Cantor set (see \cite{Bavard}).

Before moving on, we want to mention one feature of curve graphs which will be 
important for us: their hyperbolicity constant is independent of the 
surface \cite{Aougab, Bowditch, CRS, Unicorns, Bicorns}. Similar results hold for some of the
variants of curve graphs (see below).

\subsection*{A new curve graph}
Traditionally, curve graphs have vertices corresponding to 
\emph{isotopy classes} of simple closed curves (or similar objects) on surfaces and are thus by definition unable to capture any information about $\Diff_0$.

In this article we instead study the graph
$\cd(S)$ whose vertices correspond to the actual (essential) simple closed
curves on $S$ (not their isotopy classes). An edge connects two
vertices when the corresponding curves are disjoint.

In order to study $\cd(S)$ we relate its geometry to the geometry of
curve graphs of the surface $S$ punctured at finite sets $P \subset S$.
For technical reasons, we need to use so-called surviving curve graphs
$\mathcal{C}^s(S-P)$ here (compare Remark~\ref{rem:why-surviving}).
The key tool is given by Lemma~\ref{lem:realise-full} which shows that the distance between
two vertices in $\cd(S)$ can be computed using the distance in the
surviving curve graph $\mathcal{C}^s(S-P)$ of the punctured
surface $S-P$, provided the puncture set $P$ is chosen correctly (depending on the 
two vertices).

Just like for usual curve graphs, these surviving curve graphs $\mathcal{C}^s(S-P)$ are
hyperbolic with a constant independent of the choice of $P$, as was shown by A.~Rasmussen \cite{Rasmussen}.
Combining this with the distance estimate, hyperbolicity of $\cd(S)$ follows.


\subsection*{Building quasi-morphisms}  At this point we have an action of $\Diff(S_g)$ on
the hyperbolic graph $\cd(S)$. In order to produce quasi-morphisms
using \cite{BF} we now need to find diffeomorphisms which act {\em hyperbolically} (i.e. with positive asymptotic translation length) and that are \emph{independent} (i.e. there is a bound on how far their axes fellow travel even after applying any diffeomorphism to either axis). We refer the reader to Section~\ref{sec:bg} for details on these notions.


We remark that in most applications of the Bestvina--Fujiwara construction \cite{BF} the independence of
elements is guaranteed by showing that the action in question satisfies \emph{WPD}. We emphasise that this is
not the case here because the stabiliser of any finite collection of points
in $\cd(S)$ contains a copy of the diffeomorphism group of a disk. In fact more is true. There is no action of $\Diff_0(S)$ satisfying \emph{WPD}: if this were the case, then $\Diff_0(S)$ would admit a non-elementary acylindrical action on a hyperbolic space \cite{Osin} and therefore have uncountably many normal subgroups \cite{DGO}. 
But $\Diff_0(S)$ is known to be simple (since it is perfect, and \cite{Eps} shows that the commutator subgroup is simple).

\subsection*{Constructing independent hyperbolic elements}
Our hyperbolic elements will be constructed using \emph{point-pushing pseudo-Anosov} mapping classes. These
are isotopically-trivial diffeomorphisms of $S$ which fix a set of points $P$ but are pseudo-Anosov
as mapping classes of $S-P$ (compare Section~\ref{sec:curve-graphs}). Using the connection between $\cd(S)$ and $\mathcal{C}^s(S-P)$ described above we
 show that any such map acts hyperbolically on $\cd(S)$, see Lemma~\ref{lem:pas-hyperbolic}.

We then construct two elements which are independent for the action of $\Diff_0(S)$ (see Theorem~\ref{thm:maineasy}). 
The key observation is that there is a Lipschitz projection map $p \colon \cd(S) \to \mathcal{C}(S)$ which is equivariant
with respect to the action of $\Diff(S)$ on both spaces. Clearly $\Diff_0(S)$ acts trivially on
the usual curve graph $\mathcal{C}(S)$. Therefore axes in $\cd(S)$ of hyperbolic elements in $\Diff_0(S)$  are (coarsely) contained in the fibers of $p$.

Hence, if one takes a diffeomorphism $\varphi$ with positive translation 
length on $\mathcal{C}(S)$, and any hyperbolic element $\psi\in \Diff_0(S)$, then
$\psi$ and $\varphi^k\psi\varphi^{-k}$ will be independent for large $k$ because their axes project to points in $\mathcal{C}(S)$ which are far apart, and this distance cannot be changed by the action of $\Diff_0(S)$. This then yields the desired independent elements, see Section~\ref{sec:proofmainthm} for details. Our method also provides unbounded quasi-morphisms on the group of Hamiltonian diffeomorphisms, see Theorem~\ref{thm:area_pres}. 

 \subsection*{The case of homemorphisms}  In general a quasi-morphism on a topological group need not be continuous, since one can always add a discontinuous bounded  function to any given quasi-morphism. However for {\em homogeneous} quasi-morphisms on $\Diff_0(S_g)$ automatic continuity does indeed hold. This fact is due to Kotschick, but a proof unfortunately did not appear in the published version of \cite{Kot} (cf.\ also \cite{EPP}) and therefore we give a proof in Appendix~\ref{sec:automatic_cont}.  This then easily implies the existence of non-trivial unbounded quasi-morphisms on the identity component of surface homeomorphisms:

 
 \begin{theorem}\label{intro:main-theorem-homeo}
  For $g\geq 1$ the space $\widetilde{\mathrm{QH}}(\Homeo_0(S_g))$ of
  unbounded quasi-morphisms on $\Homeo_0(S_g)$ is
  infinite dimensional.
\end{theorem}
\noindent As discussed below, one could also prove the above theorem by applying our methods directly to a modified version of $\cd(S)$ where we allow all continuous (not necessarily smooth) simple closed curves on $S$.

Furthermore, a closer examination of the continuity yields an equicontinuity property for homogeneous quasi-morphisms of bounded defect. Thus an application of Bavard Duality implies that the stable commutator length function is continuous as well.
\begin{theorem}[Continuity of scl]
The stable commutator length function on the group $\Diff_0(S_g)$ is continuous with respect to the $C^0$-topology.
\end{theorem}


\subsection*{Further directions} 
First, we remark that although we construct elements $\varphi\in\Diff_0(S_g)$ 
which have positive stable commutator length \emph{as elements of the group $\Diff_0(S_g)$},
it is not clear if they actually have $\scl>0$ in the full diffeomorphism 
group $\Diff(S_g)$. In any case, the argument given in this paper does not construct
independent elements of $\Diff(S_g)$ (they are even conjugate, by
construction).  In subsequent work, we will develop more robust
geometric tools that enable us to construct elements of $\Diff_0(S_g)$
which are independent even in $\Diff(S_g)$, and consequently, provide
quasi-morphisms on $\Diff(S_g)$ which are unbounded on $\Diff_0(S_g)$.
 
\smallskip Secondly, we emphasise that the tools used to prove
Theorem~\ref{intro:main-theorem} could also \emph{directly} be applied
to groups of homeomorphisms, showing a version of
Theorem~\ref{intro:main-theorem-homeo} for homeomorphism groups. This
is pertinent, since there has been recent progress in understanding
diffeomorphism groups of manifolds (and their subgroups) in the Zimmer
program (see e.g. \cite{BDZ, BFH}) whereas homeomorphism groups have
remained largely mysterious, even in the case of surfaces. The
geometric tools we introduce 
might prove to be useful to study other questions about
homeomorphism groups of surfaces.

\subsection*{Acknowledgements} 
We would like to thank Danny Calegari, Emmanuel Militon, Leonid Polterovich, Pierre Py, and Henry Wilton for interesting and helpful comments and suggestions.

The second and third author would like to thank the conference \emph{Aspects of Non-Positive and Negative Curvature in Group Theory} at the \emph{Centre International De Rencontres Math\'ematiques} where some of the work was carried out. The third author is supported by the EPSRC Fellowship EP/N019644/2.

\section{Background}
\label{sec:bg}

\subsection{Hyperbolic geometry, hyperbolic elements, and axes} \label{sec:hyperbolic-geometry}


%
In what follows, it will be most convenient to define hyperbolicity via the four-point condition. 

\begin{definition} \label{defn:fourpointcondition}
	For points $x,y,w$ of a metric space $(X,d)$ the Gromov product is defined to be 
	\[\langle x,y \rangle_w \coloneq \frac{1}{2}(d(w,x)+d(w,y)-d(x,y)).\] 
	We say that $X$ is $\delta$-hyperbolic if for all $w,x,y,z\in X$ we have that 
	\[ \langle x,z \rangle _w \geq \min\{ \langle x,y \rangle_w, \langle y,z \rangle_w\} - \delta.\]
\end{definition}
We refer the reader to \cite{AlonsoEtAl} for various other definitions
of hyperbolicity (e.g. in terms of slim or thin triangles), and proofs
of their equivalence.

\medskip In this paper, we will view graphs as metric spaces where the
length of each edge is equal to $1$. 
Unless stated otherwise, we also
assume that all graphs in this article are connected, and all actions
on graphs are by simplicial isometries.


Now suppose a group $G$ acts on a graph $\Gamma$. For $g\in G$ we define 
\[|g| \coloneq \lim_{k\to\infty} \frac{1}{k} d(x,g^k x),\] to be the \emph{asymptotic translation length} of $g$.
 We say that $g$ is a \emph{hyperbolic element} if $|g|>0$. If $g$ is a hyperbolic element then any orbit of $g$
 is a \emph{$C$-quasi-axis}, i.e. a $g$-invariant $C$-quasi-geodesic, for some $C$ depending on the orbit. For details on these notions, we refer the reader to  \cite[Chapter~III.H.1]{BridsonH}.

\subsection{Quasi-morphisms and the fragmentation norm}\label{sec:quasi-morphisms}
 A map $\varphi \colon G \to \mathbb{\R}$ is called a \emph{quasi-morphism} (of defect $D(\varphi)$) if
 \[ \sup_{g,g' \in G}|\varphi(gg')-\varphi(g)-\varphi(g')| =
 D(\varphi) < \infty. \] 

 We denote by $\widetilde{\mathrm{QH}}(G)$ the space of unbounded
 quasi-morphisms modulo the subspace of bounded functions (which are
 also quasi-morphisms). Equivalently, $\widetilde{\mathrm{QH}}(G)$ can
 be identified with the space of \emph{homogeneous} quasi-morphisms
 (i.e. those with $\varphi(g^k) =  k\varphi(g)$ for all integers $k\in \Z$ and $g\in G$)

As a quasi-morphism takes uniformly small values on individual commutators (depending
only on the defect and its value on $1$), we obtain:
\begin{lemma}\label{lem:uniform_QM}
  A uniformly perfect group does not admit an unbounded quasi-morphism.
\end{lemma}


\begin{definition}[Fragmentation Norm]
For a closed manifold $M$ of dimension $n$, it is well known that any diffeomorphism $f \in \Diff_0(M)$ can be written as a product of diffeomorphisms supported on balls (see\ eg.\ \cite{Mann}). Such a factorisation is called a {\em fragmentation}. We define the fragmentation norm
$$\| f\|_{Frag} = \min \{N \ | \ f = h_1 \cdots h_N \ ,  \ \mathrm{supp}(h_i) \subset U_i \cong B^n \}.$$
\end{definition}
In \cite{BIP} it is shown that the fragmentation norm is universal in the sense that it coarsely bounds any conjugacy-invariant norm on $\Diff_0(M)$.
Let $G$ be a perfect group (for example $\Diff_0(M)$ where $M$ is a closed manifold). Then the {\em commutator length} is defined to be
$$\mathrm{cl}(g)  = \min \{N \ | \ g = [f_1,h_1] \cdots [f_N,h_N] \},$$
where $[f,h]$ denotes the commutator of two elements in $G$. As commutator length is a
conjugacy-invariant norm, we get:
\begin{corollary}\label{cor:frag_unbounded}
  If  $\Diff_0(M)$ is not uniformly perfect, then the fragmentation norm is
  unbounded.
\end{corollary}
\noindent In view of these results, one can then deduce the fact that the fragmentation norm on $\Diff_0(S)$ for any closed surface of genus at least one (cf.\ Corollary~\ref{cor:frag})  once we have proven the existence of unbounded quasi-morphisms  as in Theorem~\ref{intro:main-theorem}. 

\subsection{Actions on hyperbolic spaces and counting quasi-morphisms}
Consider an isometric action of a group $G$ on a $\delta$-hyperbolic
graph considered with the path metric $(X,d_X)$. Epstein and Fujiwara \cite{EpsteinFujiwara, Fuji}
described certain ``counting quasi-morphisms". These generalise the
counting quasi-morphisms of Brooks \cite{Brooks} for free groups,
whereby one counts non-overlapping copies of some word. While Fujiwara
assumes in \cite{Fuji} that the group action is properly
discontinuous, this is not required for the results we use (compare
also \cite{BF} for a discussion of this point). We follow the notation of \cite{BF}.

\begin{definition}[{\cite{BF}}] \label{def:thicksim} Let $g_1,g_2\in
  G$ be two hyperbolic elements with $(K,L)$-quasi-axes $A_1$ and $A_2$.
  We write $g_1\thicksim g_2$ if for any arbitrarily long
  subsegment $J$ in $A_1$ there is an element $h\in G$ such that $hJ$
  is within the $B(K,L,\delta)$-neighborhood of $A_2$. 
\end{definition}
\noindent  Let us briefly discuss the definition above and the definition of $B=B(K,L,\delta)$. If we instead fix a large constant $B'$ (possibly depending on $g_1$ and $g_2$) and insist there are arbitrarily long subsegments (of the $(K,L)$-quasi-geodesics) that have Hausdorff distance at most $B'$, then by general properties of quasi-geodesics in Gromov hyperbolic spaces, there are arbitrarily long subsegments that have Hausdorff distance at most $B=B(K,L,\delta)$. It is this observation that enables one to show that  $\thicksim$ is an equivalence relation (see \cite{BF} for details).

When $g_1 \nsim g_2$ we also say that $g_1$ and $g_2$ are \textit{independent}. We then have the following general criterion for the existence of unbounded quasi-morphisms given an isometric action on a Gromov hyperbolic space.

\begin{theorem}[{\cite[Theorem 1 and Proposition 2]{BF}}]\label{thm:BF_infinite}
Suppose that $g_1, g_2 \in G$ act hyperbolically on $X$ and that $g_1 \not \sim g_2$. Then the space of unbounded, homogeneous quasi-morphisms on $G$ is infinite dimensional.
\end{theorem}
In particular, Bestvina and Fujiwara construct an explicit word $w$ in $g_1$ and $g_2$, and a quasi-morphism
which is unbounded on the cyclic subgroup generated by $w$.

\subsection{Curve graphs and pseudo-Anosovs}\label{sec:curve-graphs}In this section we collect some basic results on (usual) curve graphs. 

Throughout we denote by $\mathcal{C}(S)$ the \emph{curve graph} of the (finite-type) surface $S$. That is, vertices of  $\mathcal{C}(S)$ correspond to isotopy classes of essential (i.e. not nullhomotopic) and non-peripheral (i.e. not homotopic to a puncture) simple
closed curves on $S$. Edges join distinct vertices if they admit disjoint representatives. In the case of the torus or the once-punctured torus we also have edges when the classes admit representatives intersecting at most once.

Let $n\geq 0$ and $S$ be the surface of genus $g\geq 1$ with $n$ punctures. We denote by $\mathcal{C}^s(S)$ the \emph{surviving curve graph} of the surface $S$, which is the subgraph spanned by all vertices of curves that continue to be essential even in the closed surface obtained by filling in the punctures of $S$. When $g=1$ the vertices of $\mathcal{C}^s(S)$ are precisely the curves with one complementary component (i.e. \textit{non-separating}) but we also include edges between distinct vertices if they admit representatives that intersect at most once. 
Using standard surgery techniques, it is not hard to show that the graphs above are path-connected. 

As we frequently need to use both actual curves and isotopy
classes we adopt the following notational convention.
\begin{convention*}We use Greek letters for actual simple closed curves
  on $S$ and Latin letters for isotopy classes. Furthermore all curves are smooth.\end{convention*}

For two curves $a$ and $b$ we may define the \emph{geometric intersection number} $i(a,b)$ of $a$ and $b$ to be the minimal possible value of $|\alpha\cap\beta|$ where $\alpha$ and $\beta$ are transverse, and, are representatives of the isotopy classes $a$ and $b$ respectively. Therefore $i(a,b)=0$ if and only if $a$ and $b$ are adjacent vertices.

When $\alpha$ is an essential and non-peripheral simple closed curve on $S$ and $P\subset S$ is a set of points
disjoint from $\alpha$ we denote by $[\alpha]_{S-P}$ the isotopy
class defined by $\alpha$ on $S-P$. 

	For a pair of transverse curves $\alpha$ and $\beta$ disjoint from a finite subset $P\subset S$, we say
	that $\alpha$ and $\beta$ are in \emph{minimal position} in $S-P$ if
	$|\alpha\cap\beta|$ is minimal among the representatives of $[\alpha]_{S-P}$ and $[\beta]_{S-P}$.

        Minimal position can be tested in the following way. A \emph{bigon}
	of $\alpha$ and $\beta$ in $S-P$
	is a complementary region of $\alpha\cup\beta$ in $S-P$ that is homeomorphic to a disk and bounds
	exactly one subarc of $\alpha$ and one subarc of $\beta$.

\begin{lemma}[{Bigon Criterion \cite[Proposition~1.7]{Primer}}] \label{lem:bigoncriterion} 
	For transverse simple closed curves $\alpha$ and $\beta$ we have that $\alpha$ and $\beta$ are in minimal position in $S-P$
	if and only if there are no bigons of $\alpha$ and $\beta$ in $S-P$.
\end{lemma}

Finally we need two results on \emph{pseudo-Anosov mapping classes}.
We refer the reader to \cite{Primer} for a discussion
of some basic properties of pseudo-Anosov mapping classes and further references. 
The first result we require states that a pseudo-Anosov mapping class
acts on $\mathcal{C}(S)$ as a hyperbolic isometry \cite{MM1}.

The pseudo-Anosov maps we use are built as \emph{point-pushing maps}.
Recall that if $S$ is a surface of genus $g\geq 2$, and $p\in S$ is
a point, then there is a \emph{Birman exact sequence}
\[ 1 \to \pi_1(S,p) \to \mathrm{Mcg}(S-p) \to \mathrm{Mcg}(S) \to 1 ,\]
and elements of the kernel are called  \emph{point-pushing
  maps}.  We refer the reader to \cite[Section~4.2]{Primer} for details. 

Many of these point-pushing maps are in fact pseudo-Anosov. By a result of
Kra \cite{Kra} the element in $\mathrm{Mcg}(S-p)$ corresponding to $\gamma \in  \pi_1(S,p)$
is pseudo-Anosov precisely when $\gamma$ is \emph{filling} i.e. there is no non-trivial homotopy class $\alpha$ such that, after homotopies of both curves, $\alpha$ and $\gamma$ are disjoint. This is the second result about pseudo-Anosovs that we require.

\section{A new curve graph}
\label{sec:hyperbolicity}

Throughout this section we let $S$ be a closed oriented surface of genus $g \geq 1$.

\begin{definition}
 Let $\cd(S)$ be the graph whose vertices correspond to
    essential simple closed curves in $S$. Two such vertices are
    joined by an edge precisely when the corresponding curves are
    disjoint.
  We denote by $d^\dagger$ the distance in $\cd(S)$.

When $S$ is the $2$-torus we change the definition of the edges of $\cd(S)$: two vertices are joined by an edge precisely when the corresponding curves intersect at most once.
\end{definition}

It is not hard to check that $\cd(S)$ is path-connected: given the connectedness of the ordinary curve graph it suffices to consider the case when the curves are isotopic and this reduces to the case that the curves are arbitrarily close and are evidently disjoint from a common push-off. 

\begin{remark}
  For this paper, we assume that all curves in the definition of
  $\cd(S)$ are smoothly embedded. Thus $\mathrm{Homeo}(S)$
  does not act on this graph. However all our arguments would remain
  valid for $\mathrm{Homeo}(S)$
  if we define the graph $\cd(S)$ using topologically
  embedded curves instead. In fact the resulting graphs are quasi-isometric. As this article is
  mainly concerned with the case of diffeomorphism groups we will not
  elaborate on this point further.
\end{remark}


The natural action of $\Diff(S)$ on the set of curves induces an isometric action on $\cd(S)$ and by restriction we have a natural action of $\Diff_0(S)$. We wish to show that $\cd(S)$ is hyperbolic. Though this theorem can be proved in several 
	ways, for brevity and convenience we use the four-point condition, see Definition~\ref{defn:fourpointcondition}. 

	We require the following result due to Alexander~Rasmussen \cite{Rasmussen}. Recall the definition of $\mathcal{C}^s(S)$ from Section~\ref{sec:curve-graphs}.

\begin{theorem}\label{thm:uniform-hyperbolicity}
	There is a number $\delta>0$ such that  $\mathcal{C}^s(\Sigma)$ is $\delta$-hyperbolic whenever $\Sigma$
	 is a finite-type surface with positive genus.
\end{theorem}
\begin{proof}
The main result of \cite{Rasmussen} states that there exists $\delta'$ such that $\nc(\Sigma)$ is $\delta'$-hyperbolic whenever $\Sigma$ is a finite-type surface with positive genus. Here $\mathcal{NC}(\Sigma)$ is the subgraph of $\mathcal{C}^s(\Sigma)$ consisting only of non-separating curves. When the genus is $1$ these graphs coincide. When the genus is at least $2$ it is not hard to show that this subgraph is $1$-dense and the inclusion map is an isometric embedding, therefore $\mathcal{C}^s(\Sigma)$ is also hyperbolic with a universal hyperbolicity constant $\delta$. \end{proof}


	We wish to ``approximate'' $\cd(S)$ with (usual) surviving curve graphs of finite-type surfaces $\mathcal{C}^s(S-P)$.
	For $\alpha\in\cd(S)$ disjoint from a finite subset $P\subset S$ 
	we remind the reader that we write $[\alpha]_{S-P}$ for the 
	isotopy class of $\alpha$ in $S-P$.
	The following Lemma~\ref{lem:realise-full} is key. Recall the notion of minimal position from Section~\ref{sec:curve-graphs}.

\begin{lemma} \label{lem:realise-full}
	Suppose that $\alpha,\beta\in \cd(S)$ are transverse, and that $\alpha$ and $\beta$ are in minimal 
	position in $S-P$
	where $P\subset S$ is finite and disjoint from $\alpha$ and $\beta$. Then 
	\[ d_{\mathcal{C}^s(S-P)}([\alpha]_{S-P},[\beta]_{S-P}) = d^\dagger (\alpha,\beta).\]\end{lemma}

We emphasise that $\alpha$ and $\beta$ in this lemma need not be in minimal position when seen as
curves on $S$, but only when seen as curves on $S-P$ (i.e. bigons between $\alpha$ and $\beta$ in $S$ are allowed provided they contain at least one point of $P$, in which case they are not bigons in $S-P$).

	The proof of Lemma~\ref{lem:realise-full} is a corollary of the following two lemmas, which are stated 
	in a broader context.

\begin{lemma}\label{lem:bigonlessreps}
	Suppose that $\alpha_1, \ldots, \alpha_n$ are curves that are pairwise in minimal position
	in $S-P$. Let $\beta_1, \ldots, \beta_m$ be curves that are disjoint from $P$. Then the $\beta_i$
	can be isotoped in $S-P$ such that $\alpha_1, \ldots, \alpha_n, \beta_1, \ldots, \beta_m$
	 are pairwise in minimal position in $S-P$. 
\end{lemma}

\begin{proof}
	By induction it suffices to show this for $m=1$. After an isotopy of $\beta_1$, we may assume $\beta_1$ is transverse to each $\alpha_k$.
	
	So suppose that $\alpha_1, \ldots, \alpha_n, \beta_1$ has two curves
	which are not in minimal position in $S-P$. As we assume that the $\alpha_i$ are in pairwise minimal position by Lemma~\ref{lem:bigoncriterion} (Bigon Criterion)
	there is then a bigon $B$ bounded by subarcs of $\beta_1$ and some $\alpha_j$ in $S-P$. We may assume that the bigon is innermost among all bigons between $\beta_1$ and any $\alpha_k$.
	Pushing $\beta_1$ past this bigon (as in the proof of the bigon criterion, compare \cite{Primer}) decreases $\sum i(\beta_1, \alpha_j)$ by exactly two. 
Hence this process terminates
	after finitely many steps, producing a curve isotopic to $\beta_1$ in $S-P$ which is in minimal position with respect to each $\alpha_i$. Since at each stage the curves $\alpha_i$ are fixed the lemma follows.
\end{proof}

\begin{lemma}\label{lem:moveoffP}
	Let $\alpha,\beta\in\cd(S)$ and $P\subset S$ be a finite set.
Then we may find a geodesic $\alpha=\nu_0,\ldots ,\nu_k=\beta$ such that 
	$\nu_i \cap P =\emptyset$ for all $0<i<k$.
\end{lemma}

\begin{proof} 
	Pick any geodesic $\alpha=\nu_0',\ldots, \nu_k'=\beta$ and set $\nu_0=\nu_0'$ and $\nu_k=\nu_k'$. Then inductively find a perturbation $\nu_i$ of
	 each $\nu_i'$ (for $0<i<k$) such that $\nu_i$ is disjoint from $P$ but still adjacent to $\nu_{i-1}$ and
	 $\nu_{i+1}'$. 
\end{proof}

\begin{proof}[Proof of Lemma~\ref{lem:realise-full}]
	We first prove \[ d_{\mathcal{C}^s(S-P)}([\alpha]_{S-P},[\beta]_{S-P}) \geq d^\dagger (\alpha,\beta).\] Indeed, by Lemma~\ref{lem:bigonlessreps}
	any geodesic between $[\alpha]_{S-P}$ and $[\beta]_{S-P}$ is realised by vertices $\alpha=\alpha_0,\ldots ,\alpha_k=\beta$ in $S-P$
	pairwise in minimal position. In particular $\alpha_i$ is adjacent to $\alpha_{i+1}$, and each $\alpha_i\in\cd(S)$, so we are done.

	We now prove \[ d_{\mathcal{C}^s(S-P)}([\alpha]_{S-P},[\beta]_{S-P}) \leq d^\dagger (\alpha,\beta).\] Indeed, by Lemma~\ref{lem:moveoffP} we
	can find a geodesic $\alpha=\nu_0,\ldots, \nu_k=\beta$ disjoint from $P$ then simply consider the path $[\nu_i]_{S-P}$ in $\mathcal{C}^s(S-P)$.
\end{proof}

\begin{remark}\label{rem:why-surviving}
	Lemma~\ref{lem:realise-full} is false if one would use the usual curve graphs $\mathcal{C}(S-P)$.
	In the proof this manifests itself as the problem that some vertices of $\mathcal{C}(S-P)$ cannot be represented by curves which are essential on $S$. In fact, the inclusion of $\mathcal{C}^s(S-P)$ into $\mathcal{C}(S-P)$ distorts distances by arbitrarily large amounts (compare \cite{MS} for a thorough discussion of such phenomena). So the geometry of $\cd(S)$ is really captured by the surviving curve graphs, not the usual curve graphs of
	$S-P$.
\end{remark}

\begin{theorem}\label{thm:hyperbolicity}
	The graph $\cd(S)$ is hyperbolic.
\end{theorem}

\begin{proof}
	In fact we prove that $\cd(S)$ is $(\delta+4)$-hyperbolic, where $\delta$ is the uniform constant given by 
	Theorem~\ref{thm:uniform-hyperbolicity}. More precisely we show for arbitrary vertices
	 $\mu,\alpha,\beta,\gamma\in \cd(S)$ we have that 
	\begin{equation}\label{eqn:whatwewant}\langle \alpha, \gamma \rangle_\mu \geq \min \{\langle \alpha, 
	\beta \rangle_\mu, \langle \beta, \gamma \rangle_\mu\} -\delta-4.\end{equation}

	In order to prove this we relate the vertices $\mu,\alpha,\beta,\gamma$ with vertices of some $\mathcal{C}^s(S-P)$
	 for some finite $P\subset S$. The first obstacle is to remove the pathology of pairs of vertices of $\cd(S)$ 
	that are not transverse. 

        To do so, we find vertices $\alpha',\beta',\gamma'$ with properties (\ref{item:disjoint}) and (\ref{item:transverse}) below. One should think of the primed curves as perturbations of the unprimed curves.


        We set $\mu'=\mu$.
	\begin{enumerate}[(i)]
		\item \label{item:disjoint} We have that $d^\dagger(\alpha,\alpha'),d^\dagger(\beta,\beta'),d^\dagger(\gamma,\gamma') \leq 1$, and
		\item \label{item:transverse} the vertices $\mu',\alpha',\beta',\gamma'$ are transverse.
	\end{enumerate}

	{\bf Construction of $\alpha',\beta',\gamma'$:} 
	To ensure item (\ref{item:disjoint}) we first find an $\alpha''$ disjoint from and isotopic to $\alpha$ and then 
	find a small enough perturbation $\alpha'$ of $\alpha''$, such that item (\ref{item:transverse}) holds (and similarly do 
	this for $\beta$ and $\gamma$). 

	We now choose a finite subset $P\subset S$ such that any bigon between a pair of $\mu',\alpha',\beta',\gamma'$ contains
	a point of $P$. 
	By Lemma~\ref{lem:bigoncriterion} (Bigon Criterion)  this ensures that 
	$\mu',\alpha',\beta',\gamma'$ are pairwise in minimal position in $S-P$.

	By Lemma~\ref{lem:realise-full} we have that 
	\begin{equation}\label{eqn:same} d^\dagger(\kappa',\lambda')=d_{\mathcal{C}^s(S-P)}([\kappa']_{S-P},[\lambda']_{S-P}),\end{equation} 
	moreover by property (\ref{item:disjoint}) above we see that 
	\begin{equation} \label{eqn:almostsame} |\langle \kappa', \lambda' \rangle_{\mu'}-\langle \kappa, \lambda \rangle_{\mu}| \leq 2, \end{equation}
	whenever $\kappa,\lambda\in\{\alpha,\beta,\gamma\}$. By (\ref{eqn:same}) above, Theorem~\ref{thm:uniform-hyperbolicity}, and (\ref{eqn:almostsame}) above, we obtain (\ref{eqn:whatwewant}) 
	 above, and so the theorem is proved.
\end{proof}

\section{Elements acting hyperbolically on $\cd(S)$} \label{sec:elementsactinghyp}

Let $S$ be a closed hyperbolic orientable surface and $P\subset S$ be finite. Recall the definition of asymptotic translation length $|g|$ from Section~\ref{sec:hyperbolic-geometry}. For $f\in \Mcg(S-P)$ we define $|f|$ to be the asymptotic translation length of the action of $f$ on $\mathcal{C}^s(S-P)$, and for $\varphi\in\Diff(S)$ we define $|\varphi|$ similarly via its action on $\cd(S)$. 

We now construct hyperbolic elements of $\mathrm{Diff}(S)$ on $\cd(S)$. To do this we use an important result of Masur and Minsky \cite{MM1}.

\begin{theorem} \label{thm:mm} Depending only on the topology of $S-P$ there exists $c>0$ such that for any
pseudo-Anosov $f\in\Mcg(S-P)$ we have $|f|\geq c >0$.
\end{theorem}
\begin{proof}By \cite[Proposition~4.6]{MM1} there exists $c>0$ depending only on the topology of $S-P$ such that 
\[d_{\mathcal{C}(S-P)}(f^n v, v) \geq c |n|,\] for any $v\in \mathcal{C}(S-P)$ and $n\in\Z$. The same result follows immediately for the case of $\mathcal{C}^s(S-P)$ in place of $\mathcal{C}(S-P)$ and therefore $|f|\geq c>0$. \end{proof}

Similar to our convention and notation with curves, we write $[\varphi]_{S-P}$ for the isotopy class of $\varphi\in\Diff(S)$ rel $P$. Whenever we write this, we also assert that $\varphi(P) =P$. We are now ready to state a general construction of hyperbolic elements of $\Diff(S)$ on $\cd(S)$.

\begin{lemma}\label{lem:pas-hyperbolic}
Let $P\subset S$, $f\in\Mcg(S-P)$, and $\varphi\in\Diff(S)$ be such that $\varphi(P)=P$ and
$ f=[\varphi]_{S-P}.$
Then for any $\alpha\in\cd(S)$ with $\alpha\subset S-P$ and any $i\in\mathbb{Z}$ we have that 
\begin{equation}\label{eqn:trans} d_{\mathcal{C}^s(S-P)}([\alpha]_{S-P},f^i [\alpha]_{S-P}) \leq d^\dagger(\alpha,\varphi^i\alpha). \end{equation}
Furthermore $|f|\leq |\varphi|$. In particular if $f$ is pseudo-Anosov then $\varphi$ is a hyperbolic element.
\end{lemma}
\begin{proof}
	We observe that $[\varphi^i \alpha]_{S-P}=f^i[\alpha]_{S-P}$. Given any $i\in\mathbb{Z}$ by
	Lemma~\ref{lem:moveoffP} there exists a geodesic in $\cd(S)$ connecting $\alpha$ and
	$\varphi^i\alpha$ with each vertex disjoint from $P$. Consider the sequence of 
	isotopy classes of these curves on $S-P$. This sequence is a path in $\mathcal{C}^s(S-P)$ of the same length,
	and this proves the first inequality.

Now we show that $|f|\leq |\varphi|$. Given arbitrary $i\in\Z$ and $\alpha\in\cd(S)$ we claim that
\[|f|\leq \frac{1}{i}d^\dagger(\alpha,\varphi^i\alpha).\]
If $\alpha\cap P = \emptyset$ then this is immediate by Equation~\ref{eqn:trans}. So now we assume otherwise i.e. $\alpha\cap P \neq \emptyset$. By Lemma~\ref{lem:moveoffP} there exists a geodesic $\nu_0,\ldots, \nu_k$ between $\alpha$ and $\varphi^i\alpha$ such that whenever $0<i<k$ then $\nu_i\cap P=\emptyset$. We now pick a sufficiently small perturbation $\alpha'$ of $\alpha$ about $\alpha\cap P$ in a neighborhood disjoint from $\nu_1$ and $\varphi^{-i}\nu_{k-1}$, which is possible because the latter two curves are closed subsets disjoint from $P$. Hence the $\nu_i$ also connect $\alpha'$ and $\varphi^i\alpha'$. Therefore
\[d^\dagger (\alpha', \varphi^i \alpha') \leq k=d^\dagger(\alpha,\varphi^i \alpha).\]
Since $\alpha'\subset S-P$ we obtain $i|f|\leq d^\dagger (\alpha', \varphi^i \alpha') \leq k$ as required.

Finally if $f$ is pseudo-Anosov then we have $0<|f|$ by Theorem~\ref{thm:mm} and therefore $0 < |\varphi|$ because $|f|\leq |\varphi|$.
 \end{proof}

\section{Proof of the Main Theorem} \label{sec:proofmainthm}
We now verify the criteria of Bestvina--Fujiwara  (cf.\ Theorem \ref{thm:BF_infinite}) in order to prove the
existence of unbounded quasi-morphisms on $\Diff_0(S_g)$ for any $g \geq 2$.
\begin{theorem}\label{thm:maineasy}
  For $g\geq 2$ the space $\widetilde{\mathrm{QH}}(\Diff_0(S_g))$ of
  unbounded quasi-morphisms on $\Diff_0(S_g)$ is
  infinite dimensional.
\end{theorem}
\begin{proof} 
  Consider the natural map
  \[ p : \cd(S) \to \mathcal{C}(S), \] associating to a curve its
  isotopy class. This map is clearly $1$-Lipschitz and equivariant for the natural
  actions of $\Diff(S)$ on both graphs. 

  Now consider any point $P \in S$ and let $\psi$ be a smooth
  representative of any point-pushing pseudo-Anosov of $S-P$ which
  fixes $P$. Then for any curve $\alpha$ there exists $K$ such that the
  sequence $(\psi^n\alpha)_{n\in\mathbb{Z}} = A_\psi$ is a $K$-quasi-axis
  for $\psi$. So now fix such an $\alpha$ and $K$. Furthermore observe that the image $p(A_\psi)$ is the single
  point $[\alpha]\in\mathcal{C}(S)$ because $\psi$ is isotopically trivial.

  Let $B = B(K, K, \delta)$ be the constant defining the relation $\thicksim$ as given above in Definition \ref{def:thicksim},
  see Section~\ref{sec:hyperbolic-geometry}.

  Next, choose $\varphi\in\Diff(S)$ to be any representative of a pseudo-Anosov mapping class of $S$.
  Then $\varphi^k\psi\varphi^{-k}\in\Diff_0(S)$ also.
  We can now choose $A_{\varphi^k\psi\varphi^{-k}} \coloneq \varphi^k A_\psi$
  to be a $K$-quasi-axis for $\varphi^k\psi\varphi^{-k}$,
  and so we have that $p(A_{\varphi^k\psi\varphi^{-k}}) = [\varphi^k\alpha]$.

  Since the mapping class determined by $\varphi$ acts on $\mathcal{C}(S)$ 
  as a hyperbolic isometry, we can choose $k$ large enough so that 
  \[d_{\mathcal{C}(S)}(\alpha, \varphi^k\alpha) > B.\]

  On the other hand for any element $\phi \in \Diff_0(S)$ we have
  \[ p(\phi A_\psi) = p(A_\psi), \] and thus no vertex of $\phi
  A_\psi$ is within a distance $B$ of a vertex of $A_{\varphi^k\psi\varphi^{-k}}$. Hence these elements are independent in $\Diff_0(S)$ and we may apply Theorem \ref{thm:BF_infinite}. 
\end{proof}
\noindent At this point we also recall how to prove Corollary~\ref{cor:frag} from the introduction:
\begin{corollary}
	For $g \geq 1$ the group $\Diff_0(S_g)$ is not uniformly perfect and has unbounded fragmentation norm.
\end{corollary} 
\begin{proof}
	By Theorem~\ref{thm:maineasy}, there is an unbounded quasi-morphims on $\Diff_0(S_g)$. Hence, by
	Lemma~\ref{lem:uniform_QM}, $\Diff_0(S_g)$ is not uniformly perfect. Now Corollary~\ref{cor:frag_unbounded}
	implies that the fragmentation norm is unbounded.
\end{proof}

\subsection{Area-preserving case}
In fact, the argument in the proof of Theorem~\ref{thm:maineasy} can
  also be used to find quasi-morphisms on $\Diff_0(S_g)$, which are
  unbounded when restricted to the subgroup of Hamiltonian
  diffeomorphisms (which is known to be simple by a classical result of Banyaga (cf.\ \cite[Theorem 10.25]{McSa}).
\begin{theorem}\label{thm:area_pres}
  For $g\geq 2$ the space of unbounded quasi-morphisms on $\textrm{Ham}(S_g)$ that extend to $\Diff_0(S_g)$ is
  infinite dimensional.
\end{theorem}
 \begin{proof} 
 	This is proved essentially like Theorem~\ref{thm:maineasy}. We explain how to ensure that the maps
 	chosen in that proof are Hamiltonian. 
 	We begin by taking two curves $\alpha, \beta$ so that the commutator $[\alpha,\beta]$ is filling.
 	A standard argument using Moser's lemma shows that we can find representatives of the point pushes
 	$P_\alpha, P_\beta$ which preserve a common area form $\omega$, and which both fix $P$. Then $\psi$ is 
    a representative of a pseudo-Anosov on $S-P$, and is a Hamiltonian diffeomorphism of $S$
    (as commutators of area preserving surface diffeomorphisms that are isotopic to the identity are
 	Hamiltonian; cf.\ \cite[Theorem 10.12]{McSa}). 
 	
 	Using Moser's lemma again, we can choose a representative of $\varphi$ which also preserves $\omega$.
 	Then, the conjugate of $\psi$ by $\varphi^k$ is also Hamiltonian. We can now conclude the argument precisely as in Theorem~\ref{thm:maineasy}, noting that the quasi-morphisms given by Theorem~\ref{thm:BF_infinite} are 
 	unbounded on the cyclic subgroup generated by the two independent (Hamiltonian) elements $\psi, \varphi^k\psi\varphi^{-1}$.
	 \end{proof}

\subsection{The genus one case} 
\label{sec:torus}
We are left with proving the main theorem in the case where $S$ is the
$2$-torus. The proof idea is essentially the same, but the details are
different in certain places. Here we sketch the necessary
changes in the proofs.


In Section~\ref{sec:hyperbolicity} we defined $\cd(S_1)$ and proved it is hyperbolic so no changes are required. In Section~\ref{sec:elementsactinghyp} Lemma~\ref{lem:pas-hyperbolic} goes through verbatim. However in Theorem~\ref{thm:mm} one needs to be more careful because $\mathcal{C}^s(S_1-P)$ has edges when curves have intersection at most one, whereas $\mathcal{C}(S-P)$ has edges when the curves are disjoint (for $|P|\geq 2$). However it is always true that if we add edges to $\mathcal{C}(S-P)$ whenever curves intersect at most once then the graphs are quasi-isometric. To see this, observe that two curves that intersect at most once are in fact distance at most $2$ in $\mathcal{C}(S-P)$. Therefore the statement of Theorem~\ref{thm:mm} is true for the torus.

In the proof of the main theorem for hyperbolic surfaces, we use point-pushing pseudo-Anosov mapping classes. On the torus we require $|P|\geq 2$ for the existence of such pseudo-Anosovs, and we can choose $\psi$ to be a point-pushing pseudo-Anosov on
$S- P$ by pushing one point. The rest of the argument
applies as before.

\subsection{The genus zero case}\label{sec:spherefail}
At this point, we quickly emphasise which parts of our proof strategy for $\Diff_0(S^2)$ break down (as it must, given that the fragmentation norm is bounded on the sphere). 
With our previous definition, $\cd(S^2)$ has no vertices because we require essential curves, and any curve on $S^2$ is inessential. If we were to modify the definition to allow non-essential curves, it is straightforward to see that $\cd(S^2)$ has diameter $2$. Hence, all elements of $\Diff_0(S^2)$ would act elliptically on that graph, and so \cite{BF} could not be invoked. 

\subsection{Quasi-morphisms on $\textrm{Homeo}_0(S)$}

By a result of Kotschick (cf.\ Theorem \ref{thm:automatic_cont_surface}) homogeneous quasi-morphisms on $\Diff_0(S)$ are automatically continuous with respect to the $C^0$-topology. Moreover, a brief inspection of the proof of Theorem \ref{thm:automatic_cont} shows that any homogeneous quasi-morphism extends to the $C^0$-closure of $\Diff_0(S)$ inside $\mathrm{Homeo}_0(S)$.  Since any homeomorphism on a closed surface can be uniformly approximated by diffeomorphisms (cf.\ \cite{Munk}), this then gives the desired extension to $\mathrm{Homeo}_0(S)$.
\begin{corollary}\label{cor:main-theorem}
  The space of unbounded quasi-morphisms on $\mathrm{Homeo}_0(S)$ is infinite dimensional for any closed surface $S$ of genus at least one.
\end{corollary}
\noindent Alternatively one could actually construct quasi-morphisms directly on the group $\mathrm{Homeo}_0(S)$ by considering a curve graph of {\em topologically embedded} curves with the obvious edge relation. The arguments used to prove Theorem \ref{thm:maineasy} readily extend, although some care is needed in dealing with topological transversality, minimal position and so on. 

A closer inspection of the proof of Theorem \ref{thm:automatic_cont} actually shows that the set of homogeneous quasi-morphisms of bounded defect is {\em point-wise equicontinous} which in view of Bavard Duality shows that the stable commutator length function is too.
\begin{theorem}\label{thm:scl_cont} The stable commutator length function $\mathrm{scl}\colon \Diff^r_0(S) \to \R$ is continuous with respect to the $C^0$-topology for any $0 \le r \le \infty$.
\end{theorem}
\begin{proof}
For a group $G$ let  $\widetilde{\mathrm{QH}}(G)$ denote the group of homogeneous quasi-morphisms and let $\widetilde{\mathrm{QH}}_1(G)$ denote the subset of defect $D(\varphi) =1$. By Bavard Duality we have
\[\mathrm{scl}(g) = \sup_{\varphi \in \widetilde{\mathrm{QH}}(G)}\frac{|\varphi(g)|}{2D(\varphi)} = \sup_{\varphi \in \widetilde{\mathrm{QH}}_1(G)}\frac{|\varphi(g)|}{2}. \]
For $G = \Diff_0^r(S) $ the family of real-valued functions $\widetilde{\mathrm{QH}}_1(\Diff_0^r(S) )$ is equicontinuous at each $g\in G$ (cf. proof of Theorem~\ref{thm:automatic_cont} below)  so it follows that the right hand side is continuous in $g$, whence we deduce that the stable commutator length function is continuous. 
\end{proof}
\begin{remark}
The results above for $\mathrm{scl}$ is not specific to surface groups. However, it is known to be vacuous in any dimension other than $2$ and possibly $4$ in view of \cite{BIP,Tsu,Tsu_even} and Theorem~\ref{intro:main-theorem}. 
\end{remark}

\appendix
\section{Automatic continuity}\label{sec:automatic_cont}

It was observed by Entov--Polterovich--Py \cite{EPP} that homogeneous quasi-morphisms of area-preserving maps of surfaces are continuous in the $C^0$-topology if they vanish on all elements supported on disks of bounded area. The fact that such a statement holds was suggested by Kotschick \cite{Kot} to whom the idea is attributed. Kotschick also observed that this continuity holds in the setting of diffeomorphism groups, where the corresponding fragmentation properties are well known. However, the proof of this fact did not appear in  \cite{Kot}.

Most of the results on continuity of quasi-morphisms in this section are not new, but as they have not previously appeared in print in the form we present them, we choose to include them. What does appear to be new is the fact that the stable commutator length function is continuous. Our arguments follow the lines of \cite{EPP}, but things are significantly simpler than in the area-preserving case they consider.

\subsection*{Bounded fragmentation} For the sake of giving a uniform account we let $\Diff^r(M)$ denote the group of $C^r$-diffeomorphisms of a closed manifold $M$, $\Diff_0^r(M)$ the identity component of the group of $C^r$-diffeomorphisms, and $\Diff_0^\infty(M) = \Diff_0(M)$ the group of smooth diffeomorphisms. The following is just the observation that the standard fragmentation procedure for diffeomorphisms on compact manifolds yields factorisations of bounded length. In the case of $C^r$-diffeomorphisms this is elementary (see e.g.\ \cite[Lemma~2.1]{Mann}) and in the case of homeomorphisms it follows from classical results of Edwards--Kirby \cite{EdKir}.
\begin{lemma}[General Case: $C^r$-topology]\label{lem:bounded_frag}
For any $1 \le r \le \infty $, there is a neighborhood $\mathcal{U}_{Id}  \subseteq \mathrm{Diff}^r_0(M)$ of the identity map with respect to the $C^r$-topology such that any $f \in \mathcal{U}_{Id}$ can be  written as a product of $C_M$ diffeomorphisms supported on open disks for some constant $C_M$ depending only on the manifold.
\end{lemma}
In the case of surfaces one can essentially model the argument of Edwards--Kirby in the smooth case to obtain a bounded fragmentation property with respect to the $C^0$-topology. For this one first notes that any disc $D$ that is displaced some small distance $\varepsilon$ by some diffeomorphism $f$ can be moved back to its original position by a diffeomorphism with support in a $C\varepsilon$-neighbourhood $N_{C\varepsilon}(D)$ of $D$ for some universal constant $C>0$ using for example \cite[Lemma 7.1]{EPP}. This then enables one to build some $f_D$ that agrees with $f$ on the disk $D$ and has support in $N_{C\varepsilon}(D)$. Moreover, if the original diffeomorphism was the identity on some open neighborhoods of points on $\partial D$, the same can be assumed of $f_D$, up to shrinking the neighbourhoods slightly. Repeated application of this procedure applied to say a handle decomposition of the surface then gives a fragmentation of bounded length. 
\begin{lemma}[Surface Case: $C^0$-topology]\label{lem:bounded_frag_surface}
There is a neighborhood $\mathcal{U}_{Id}  \subseteq \mathrm{Diff}^r_0(S)$ of the identity map with respect to the $C^0$-topology such that any $f \in \mathcal{U}_{Id}$ can be  written as a product of $C_S$ diffeomorphisms supported on open disks for some constant $C_S$ depending only on the surface. 
\end{lemma}

\subsection*{Boundedness near the identity implies continuity}
It is well known that homogeneous quasi-morphisms vanish on diffeomorphisms supported on disks. This follows for example from the fact that the group of compactly supported diffeomorphisms of the (open) unit disk is uniformly perfect (cf.\ \cite{BIP}, \cite{Kot}), \cite{Tsu}).
\begin{lemma}\label{lem:QH_disc_vanish}
Let $\varphi \colon \mathrm{Diff}^r_0(M) \to \R$ be a homogeneous quasi-morphism. Then $\varphi(g) = 0$ for any element with support contained in a ball $\mathrm{supp}(g) \subset U \cong B^n$.
\end{lemma}
We can then deduce that there is a uniform bound on some open neighborhood of the identity for {\em any} homogeneous quasi-morphism in terms of the defect. The following fact is often attributed to \cite{Shtern}.
\begin{lemma}\label{lem:bounded_small}
For any $1 \le r \le \infty$, there is a $C^r$-neighborhood $\mathcal{U}_{Id}  \subseteq \mathrm{Diff}^r_0(M)$ so that any homogeneous quasi-morphism is bounded by some constant multiple of the defect on $\mathcal{U}_{Id}$. In the case of surfaces this also holds for a $C^0$-neighborhood.
\end{lemma}
\begin{proof}
Since $f  = g_1 g_2 \cdots g_k$ can be factored as a product of $k = C_M$ diffeomorphisms supported on disks we conclude that
\begin{align*}
|\varphi(g_1 g_2 \cdots g_k)| & =|\varphi(g_1 g_2 \cdots g_k) -( \varphi( g_1) + \cdots  +\varphi( g_k)) | \\
 &\le (C_M-1) D(\varphi).
\end{align*}
Here we use the quasi-morphism property repeatedly as well as the fact that homogeneous quasi-morphisms vanish on maps that are supported on balls.
\end{proof}
\begin{theorem}\label{thm:automatic_cont}
Any homogeneous quasi-morphism $\varphi \colon \mathrm{Diff}^{r}_0(M) \to \R$ is continuous with respect to the $C^s$-topology for any $1 \le s \le r \le \infty $.
\end{theorem}
\begin{proof}
Let $f \in \mathrm{Diff}_0(M)$. Choose a neighborhood $\mathcal{V}_n$ so that for any $g \in \mathcal{V}_n$ we have that $f^n g^ {-n}$ lies in the neighborhood $\mathcal{U}_{Id}$ for some fixed $n$. Then using homogeneity and the quasi-morphism property we have that
\[n|\varphi (f) - \varphi (g)| = |\varphi (f^n) + \varphi (g^{-n})|   \le  |\varphi (f^n g^ {-n})| + D(\varphi) \le (C_S-1)D(\varphi) + D(\varphi). \]
Then dividing and letting $n \to \infty$ the continuity follows.
\end{proof}
\noindent The same argument used above shows that quasi-morphisms on surface diffeomorphism groups are continuous in the $C^0$-sense in view of Lemma \ref{lem:bounded_frag_surface}.
\begin{theorem}[Kotschick]\label{thm:automatic_cont_surface} Any homogeneous  quasi-morphism $\varphi \colon \mathrm{Diff}_0(S) \to \R$ is continuous with respect to the $C^0$-topology.
\end{theorem}

\bibliographystyle{alpha}
\bibliography{dagger}

\end{document}